\documentclass[leqno,12pt]{article}
\usepackage{amsmath, amssymb}
\usepackage{amsthm}
\theoremstyle{plain}
\newtheorem{theorem}{\indent\sc Theorem}[section]
\newtheorem{lemma}[theorem]{\indent\sc Lemma}

\newtheorem{proposition}[theorem]{\indent\sc Proposition}

\theoremstyle{definition}
\newtheorem{definition}[theorem]{\indent\sc Definition}

\title{Closure of dilates of shift-invariant subspaces}
\author{Mois\'es Soto-Bajo\\
\bigskip
{\small moises.soto@uam.es}\\
{\small Department of Mathematics, Faculty of Sciences,}\\
{\small Autonomous University of Madrid, Cantoblanco,}\\
{\small 28049-Madrid, Spain}
}
\date{}
\begin{document}

\maketitle

\begin{center}
\line(1,0){450}
\end{center}
\textbf{Abstract:} Let $V$ be any shift-invariant subspace of square summable functions. 
We prove that if for some $A$ expansive dilation $V$ is $A$-refinable, 
then the completeness property is equivalent to several conditions on the local behaviour at the origin of the spectral function of $V$, 
among them the origin is a point of $A^*$-approximate continuity of the spectral function if we assume this value to be one. 
We present our results also in the more general setting of $A$-reducing spaces. 
We also prove that the origin is a point of $A^*$-approximate continuity of the Fourier transform of any semiorthogonal tight frame wavelet if we assume this value to be zero.



\begin{center}
\line(1,0){450}
\end{center}
\section{Introduction}

A closed subspace $V$ of $L^2(\mathbb{R}^{d})$ is called shift-invariant 
if for all $f\in V$ and $k\in\mathbb{Z}^d$ it follows that $f(\cdot-k)\in V$. 
For the theory of shift-invariant subspaces see \cite{dBDVR1}, \cite{Bow2}, \cite{BR}, \cite{GuHa:05} and the references therein. 
This kind of subspaces have been used extensively in many areas of contemporary mathematics, as Approximation Theory, Finite Element Analysis or Wavelet Theory.

In addition to translations we consider also dilations. 
We say that a linear map $A:\mathbb{R}^d\rightarrow\mathbb{R}^d$ is an expansive dilation if $A(\mathbb{Z}^d)\subseteq\mathbb{Z}^d$ and the modulus of all (complex) eigenvalues is strictly greater than $1$. 
Thus the associated dilation operator is defined by
\begin{displaymath}
\mathcal{D}_{A}f(x)=d_A^{\frac{1}{2}}\,f(A\,x)
\mbox{ for }f\in L^2(\mathbb{R}^{d})\,,
\end{displaymath}
where $d_A=|det(A)|$. The adjoint of $A$ is denoted by $A^*$.

Let $A$ denote henceforth an expansive dilation. 
A shift-invariant subspace $V$ is called $A$-refinable if $V\subseteq\mathcal{D}_{A}V$. 
Note that from the condition $A(\mathbb{Z}^d)\subseteq\mathbb{Z}^d$ it follows that $\mathcal{D}_{A}V$ is also shift-invariant.

Wavelets were introduced relatively recently, in the beginning of the 80's and found applications in diverse disciplines. 
One of the main tools for constructing wavelets are the so called MRAs, introduced by Mallat and Meyer (\cite{Ma:89}, \cite{Me:90}). 
Extending this concept, in \cite{BMM1} Baggett, Medina, and Merrill introduced the generalized multiresolution analyses. 
A shift-invariant subspace $V$ is said to generate an $A$-generalized multiresolution analysis ($A$-GMRA) in $L^2(\mathbb{R}^{d})$ 
if it is $A$-refinable, $\bigcap_{j\in\mathbb{Z}}\mathcal{D}_{A}^{j}V=\{0\}$ and $\overline{\bigcup_{j\in\mathbb{Z}}\mathcal{D}_{A}^{j}V}=L^2(\mathbb{R}^{d})$. 
The subspace $V$ is called the core subspace of the $A$-GMRA.

The study of $A$-refinable shift-invariant subspaces and their approximation properties has attracted the interest of a great number of researchers in many areas at least in the last thirty years. 
The cornerstone in this study is the so called completeness property:
\begin{displaymath}
\overline{\bigcup_{j\in\mathbb{Z}}\mathcal{D}_{A}^{j}V}=L^2(\mathbb{R}^{d})\,.
\end{displaymath}
This allow us to approximate the functions in $L^2(\mathbb{R}^{d})$ by elements in the dilates of the shift-invariant subspace $V$. 
So many pages have been devoted to understand the nature of this property, and eventually to characterize the subspaces which satisfy it.

In \cite{KSa} can be found a brief review on the characterization of (frame) scaling functions, 
which particularly include results on the completeness of singly-generated $A$-refinable shift-invariant subspaces 
(we refer to \cite{Mad}, \cite{dBDVR3}, \cite{HWW}, \cite{LMS}, \cite{DGH2}, \cite{KKL4}, \cite{CKSa1}, \cite{CKSa2}, \cite{LiLi}, \cite{Cu}, \cite{KSa}, \cite{ZL}). 
See also \cite{JS1} and \cite{Ca2} for finitely generated subspaces. 
In \cite{Dut:05} a characterization of completeness in the general case was established. 
Before presenting that we introduce some definitions.

Given a shift-invariant subspace $V$ we say that $\mathcal{G}=\{\phi^{\alpha}\}_{\alpha\in I}\subseteq V$ ($I$ a set of indices) 
is a tight frame generator of $V$ if
\begin{displaymath}
\sum_{\alpha\in I}\sum_{k\in\mathbb{Z}}|<f,\phi^{\alpha}(\cdot-k)>|^2=\left\|\,f\,\right\|^{2}
\quad\forall\:f\in V\,,
\end{displaymath}
where $<>$ and $\left\|\,\cdot\,\right\|$ denote respectively the usual inner product and norm in $L^2(\mathbb{R}^{d})$. 
It is well known that every shift-invariant subspace has a tight frame generator (see \cite{Bow2}).

In this paper we adopt the convention that the Fourier transform of a function $f\in L^1(\mathbb{R}^d)\cap L^2(\mathbb{R}^{d})$ is defined by
\begin{displaymath}
\widehat{f}(\xi)=\int_{\mathbb{R}^d}f(x)\,e^{-2\pi ix\cdot\xi}\,dx\,.
\end{displaymath}

The characterization of completeness given in \cite{Dut:05} is the following:
\begin{theorem} \label{thm:characterization Dutkay}
Let $V$ be an $A$-refinable shift-invariant subspace and $\mathcal{G}:=\{\phi^{\alpha}\}_{\alpha\in I}$ a tight frame generator of $V$. 
Then the following conditions are equivalent:
\begin{itemize}
	\item[-] $\overline{\bigcup_{j\in\mathbb{Z}}\mathcal{D}_{A}^{j}V}=L^2(\mathbb{R}^{d})$;
	\item[-] one has
\begin{displaymath}
\lim_{j\to\infty}\sum_{\alpha\in I}|\widehat{\phi^{\alpha}}(A^{*-j}\xi)|^2=1
\quad a.e.\,\xi\in\mathbb{R}^d\,.
\end{displaymath}
\end{itemize}
\end{theorem}

We make some comments which arise from this result. 
Fixed a shift-invariant subspace $V$ and a tight frame generator $\mathcal{G}=\{\phi^{\alpha}\}_{\alpha\in I}$ of $V$, the expression
\begin{displaymath}
\sigma_V(\xi)=\sum_{\alpha\in I}|\widehat{\phi^{\alpha}}(\xi)|^2
\end{displaymath}
is called the spectral function of $V$, 
and defines a locally integrable function on $\mathbb{R}^d$ which does not depend on the chosen $\mathcal{G}$, only on $V$. 
It was introduced by Rzeszotnik in \cite{Rz:01} (see also \cite{WW:00}, \cite{GuHa:05}).

From all these characterizations we cited above one can extract the following conclusion: 
the completeness property depends on the behaviour at the origin of the spectral function of $V$.

In order to study the local behaviour of a function we introduce some concepts from \cite{CKSa1} 
which extend the classical notions of the density point of a measurable set and the point of approximate continuity of a measurable function (see \cite{N:60}, \cite{CKSa2}). 
$B_r$ denotes the open ball in $\mathbb{R}^d$ of radius $r>0$ around the origin.
\begin{definition} \label{defi:point of A-density}
Let $E\subseteq\mathbb{R}^d$ be a set of positive measure. 
We say that the origin is a point of $A$-density for $E$ if for all $r>0$
\begin{displaymath}
\lim_{j\to\infty}\frac{|E\cap A^{-j}B_r|}{|A^{-j}B_r|}=1\,.
\end{displaymath}
\end{definition}
\begin{definition} \label{defi:point of A-approximate continuity}
Let $f\,:\,\mathbb{R}^d\longrightarrow\mathbb{C}$ be a measurable function. 
We say that the origin is a point of $A$-approximate continuity of $f$ 
if there exists $E\subseteq\mathbb{R}^d$ of positive measure such that the origin is a point of $A$-density for $E$ and
\begin{displaymath}
\lim_{\begin{subarray}{c} x\to 0 \\ x\in E \end{subarray}}f(x)=f(0)\,.
\end{displaymath}
We also say that $f$ is $A$-locally nonzero at the origin if
\begin{displaymath}
\lim_{j\to\infty}\frac{|\{x\in A^{-j}B_1\::\:f(x)=0\}|}{|A^{-j}B_1|}=0\,.
\end{displaymath}
That is, if the origin is a point of $A$-density for $Supp(f):=\{x\in\mathbb{R}^d\::\:f(x)\neq 0\}$.
\end{definition}

Following \cite{LMS}, we make the following:
\begin{definition} \label{defi:A-absorbing}
We say that $E\subseteq\mathbb{R}^d$ measurable is $A$-absorbing in $\mathbb{R}^d$ 
if for almost every $\xi\in\mathbb{R}^d$ there exists a positive integer $j_0$ (possibly dependent on $\xi$) such that $A^{j}\xi\in E$ if $j\geq j_0$.
\end{definition}
This notion was used in \cite{LMS} for characterizing scaling functions in the dyadical case on the real line.

We summarize the known characterizations of scaling functions in $L^2(\mathbb{R}^{d})$ in the following way: 
let $\varphi\in L^2(\mathbb{R}^{d})$ and consider the corresponding principal shift-invariant subspace $V$ generated by the shifts of $\varphi$. 
Define $\phi\in L^2(\mathbb{R}^{d})$ given by
\begin{displaymath}
\widehat{\phi}(\xi):=\left\{\begin{array}{lc}
\frac{\widehat{\varphi}(\xi)}{[\widehat{\varphi},\widehat{\varphi}]^{\frac{1}{2}}(\xi)} & \mbox{if }[\widehat{\varphi},\widehat{\varphi}](\xi)\neq 0 \\ 0 & elsewhere \end{array}\right.\,,
\end{displaymath}
where $[\widehat{\varphi},\widehat{\varphi}](\xi):=\sum_{k\in\mathbb{Z}^d}|\widehat{\varphi}(\xi+k)|^2$. 
In \cite{dBDVR1} was proved that $\phi$ is a tight frame generator of $V$. 
Thus, note that the spectral function of $V$ is $\sigma_V(\xi)=|\widehat{\phi}(\xi)|^2$. 
Suppose that $V$ is $A$-refinable. 
This was shown in \cite{KKL4} and \cite{CKSa2} to be equivalent to the existence of some $m\in L^{\infty}(\mathbb{T}^d)$, called low-pass filter, such that
\begin{equation} \label{eq:scaling equation}
\widehat{\phi}(A^*\xi)=m(\xi)\,\widehat{\phi}(\xi)
\quad a.e.\,\xi\in\mathbb{R}^d\,.
\end{equation}
\begin{theorem}
The following conditions are equivalent:
\begin{itemize}
	\item[-] $\overline{\bigcup_{j\in\mathbb{Z}}\mathcal{D}_A^{j}V}=L^2(\mathbb{R}^{d})\,$;
	\item[-] one has $\bigcup_{j\in\mathbb{Z}}A^{*j}(Supp(\widehat{\varphi}))=\mathbb{R}^d$, 
	or equivalently $\bigcup_{j\in\mathbb{Z}}A^{*j}(Supp(\widehat{\phi}))=\mathbb{R}^d\,$;
	\item[-] for any bounded set $E$ of positive measure
	\begin{displaymath}
	\lim_{j\to\infty}\frac{1}{|A^{*-j}E|}\,\int_{A^{*-j}E}|\widehat{\phi}(\xi)|^2\,d\xi=1\,;
	\end{displaymath}
	\item[-] the set $Supp(\widehat{\phi})$ is $A^{*-1}$-absorbing in $\mathbb{R}^d$;
	\item[-] $\lim_{j\to\infty}|\widehat{\phi}(A^{*-j}\xi)|>0$ for almost every $\xi\in\mathbb{R}^d$;
	\item[-] $\lim_{j\to\infty}|\widehat{\phi}(A^{*-j}\xi)|=1$ for almost every $\xi\in\mathbb{R}^d$;
	\item[-] $|\widehat{\phi}|$ is $A^*$-locally nonzero at the origin;
	\item[-] the origin is a point of $A^*$-approximate continuity of $|\widehat{\phi}(\xi)|$ if we put $|\widehat{\phi}(0)|=1$.
\end{itemize}
\end{theorem}
For a proof of the different equivalences in the above result we refer to 
\cite{Mad}, \cite{dBDVR3}, \cite{HWW}, \cite{LMS}, \cite{DGH2}, \cite{KKL4}, \cite{CKSa1}, \cite{CKSa2}, \cite{LiLi}, \cite{Cu}, \cite{KSa}, \cite{ZL}.

In this paper we prove the following result about the characterization of completeness property, 
which includes Theorem \ref{thm:characterization Dutkay} and generalizes all the above-cited results.
\begin{theorem} \label{thm:characterizations for L2}
Let $V$ be an $A$-refinable shift-invariant subspace. 
Then the following conditions are equivalent:
\begin{itemize}
	\item[-] $\overline{\bigcup_{j\in\mathbb{Z}}\mathcal{D}_A^{j}V}=L^2(\mathbb{R}^{d})\,$;
	\item[-] if $\mathcal{G}$ generates $V$ in the sense that the linear span of the shifts of $\mathcal{G}$ is dense in $V$, 
then $\bigcup_{\phi\in\mathcal{G}}\bigcup_{j\in\mathbb{Z}}A^{*j}(Supp(\widehat{\phi}))=\mathbb{R}^d\,$;
	\item[-] for any bounded set $E$ of positive measure
	\begin{displaymath}
	\lim_{j\to\infty}\frac{1}{|A^{*-j}E|}\,\int_{A^{*-j}E}\sigma_V(\xi)\,d\xi=1\,;
	\end{displaymath}
	\item[-] the set $Supp(\sigma_V)$ is $A^{*-1}$-absorbing in $\mathbb{R}^d$;
	\item[-] $\lim_{j\to\infty}\sigma_V(A^{*-j}\xi)>0$ for almost every $\xi\in\mathbb{R}^d$;
	\item[-] $\lim_{j\to\infty}\sigma_V(A^{*-j}\xi)=1$ for almost every $\xi\in\mathbb{R}^d$;
	\item[-] $\sigma_V$ is $A^*$-locally nonzero at the origin;
	\item[-] the origin is a point of $A^*$-approximate continuity of $\sigma_V$ if we put $\sigma_V(0)=1$.
\end{itemize}
\end{theorem}

The previous theorem provides us with a set of different conditions on the spectral function, 
all of them equivalent to the completeness property of an $A$-refinable shift-invariant subspace $V$. 
However, if $V$ does not satisfy the completeness property does not mean that it is useless for approximating functions. 
This more general situation has been considered in many papers 
(\cite{Mad}, \cite{DL}, \cite{LMS}, \cite{DLS2}, \cite{KL}, \cite{GuHa2}, \cite{DDG}, \cite{DDGH}, \cite{DDGH3}, \cite{LiLi}, \cite{KSa}, \cite{ZL}).

We will say that a shift-invariant subspace $H$ of $L^2(\mathbb{R}^{d})$ is an $A$-reducing space if $\mathcal{D}_{A}H=H$. 
It is easy to see that, if $V$ is an $A$-refinable shift-invariant subspace, 
then $\overline{\bigcup_{j\in\mathbb{Z}}\mathcal{D}_{A}^{j}V}$ is an $A$-reducing space. 
In this way, $A$-reducing spaces arise naturally in the study of the dilates of shift-invariant subspaces. 
Apart from $L^2(\mathbb{R}^{d})$, the first $A$-reducing space considered in this context was the Hardy space $H^2$ on the real line, with dyadic dilations 
(see \cite{HKLS}, \cite{Au2}, \cite{HWW}, \cite{HW:96}). 
It is known that the $A$-reducing subspaces of $L^2(\mathbb{R}^{d})$ have the following form
\begin{displaymath}
H^2_G=\{f\in L^2(\mathbb{R}^{d})\::\:Supp\widehat{f}\subseteq G\}
\end{displaymath}
for some $G\subseteq\mathbb{R}^d$, where $G$ is an $A^*$-invariant measurable set. 
Here $A^*$-invariant means that $A^*G=G$ (see \cite{DDGH}, \cite{KSa}).

Fixed an $A^*$-invariant set $G$, a shift-invariant subspace $V$ is said to generate an $A$-generalized multiresolution analysis ($A$-GMRA) in $H^2_G$ 
if it is $A$-refinable, $\bigcap_{j\in\mathbb{Z}}\mathcal{D}_{A}^{j}V=\{0\}$ and $\overline{\bigcup_{j\in\mathbb{Z}}\mathcal{D}_{A}^{j}V}=H^2_G$ (see \cite{BMM1}). 
In this note we will study the completeness property
\begin{displaymath}
\overline{\bigcup_{j\in\mathbb{Z}}\mathcal{D}_{A}^{j}V}=H^2_G\,.
\end{displaymath}

In section \ref{sec:tools} we develope some tools which will be necessary in the following sections. 
In section \ref{sec:completeness} we provide several characterizations which generalize Theorem \ref{thm:characterizations for L2}. 
For a review on the intersection property $\bigcap_{j\in\mathbb{Z}}\mathcal{D}_{A}^{j}V=\{0\}$ we refer to \cite{Bow5}. 
In section \ref{sec:wavelets} we make some related comments on the local behaviour of the Fourier transform of the multiwavelets.

\section{Tools} \label{sec:tools}

In this section we develope some tools which will be necessary in the following sections. 
We summarize some properties of the spectral function in the next proposition 
(see \cite{BR}, \cite{GuHa:05}). 
We use the symbol $\oplus$ for orthogonal direct sums of subspaces.

\begin{proposition} \label{prop:properties spectral function}
Let $V,\{V_n\}_{n=1}^{\infty}$ be shift-invariant subspaces of $L^2(\mathbb{R}^{d})$, and $G$ an $A^*$-invariant set. 
Then the following properties hold.
\begin{description}
	\item[(i)] $\sigma_{L^2(\mathbb{R}^{d})}=1$ and $\sigma_{\{0\}}=0$ $a.e.$;
	\item[(ii)] If $V=\bigoplus_{n=1}^{\infty}V_n$, 
	then $\sigma_V=\sum_{n=1}^{\infty}\sigma_{V_n}$;
	\item[(iii)] Assume $V_1\subseteq V$. 
	Then $\sigma_{V_1}\leq\sigma_V$  $a.e.$, and $\sigma_{V_1}=\sigma_V$ $a.e.$ if and only if $V_1=V$;
	\item[(iv)] $\sigma_{H^2_G}={\text{\Large $\chi$}}_G$ and 
	$V\subseteq H^2_G$ if and only if $\sigma_V\leq{\text{\Large $\chi$}}_G$ $a.e.$;
	\item[(v)] $0\leq\sigma_V\leq 1$ $a.e.$;
	\item[(vi)] If $\mathcal{F}\subseteq V$ generates $V$, then
\begin{equation} \label{eq:support of spectral function}
Supp(\sigma_V)=\bigcup_{\phi\in\mathcal{F}}Supp(\widehat{\phi})\,;
\end{equation}
	\item[(vii)] The spectral function of the shift-invariant subspace $\mathcal{D}_{A}V$ is
\begin{displaymath}
\sigma_{\mathcal{D}_AV}(\xi)=\sigma_V(A^{*-1}\xi)\quad a.e.\,.
\end{displaymath}
\end{description}
\end{proposition}
\begin{proof}
{The proof of (i), (ii), (iii), (v) and (vii) can be found in \cite{BR}. 
We prove (iv). 
Firstly, if for any integer $m$ we define $\phi^m$ by the identity $\widehat{\phi^m}:={\text{\Large $\chi$}}_{([0,1)^d+m)\cap G}$, 
we have $\phi^m\in H^2_G$. 
It is easy to see that $\{(\phi^m(\cdot-k))^{\,\widehat{}\,}\}_{k,m\in\mathbb{Z}^d}$ is a tight frame of $L^2(G)$ 
(note that $(\phi^m(\cdot-k))^{\,\widehat{}}(\xi)=e^{-2\pi ik\cdot\xi}\,{\text{\Large $\chi$}}_{([0,1)^d+m)\cap G}(\xi)$) 
so $\{\phi^m\}_{m\in\mathbb{Z}^d}$ is a tight frame generator of $H^2_G$. 
Thus
\begin{displaymath}
\sigma_{H^2_G}(\xi)=
\sum_{m\in\mathbb{Z}^d}|\widehat{\phi^m}(\xi)|^2=
\sum_{m\in\mathbb{Z}^d}{\text{\Large $\chi$}}_{([0,1)^d+m)\cap G}(\xi)=
{\text{\Large $\chi$}}_G(\xi)\,.
\end{displaymath}
If $V\subseteq H^2_G$, by (ii) we have $\sigma_V(\xi)\leq{\text{\Large $\chi$}}_G(\xi)$ for almost every $\xi\in\mathbb{R}^d$. 
Reciprocally, if $\mathcal{G}$ is a tight frame generator of $V$ and $\sigma_V(\xi)\leq{\text{\Large $\chi$}}_G(\xi)$, 
then any $\phi\in\mathcal{G}$ satisfies $\widehat{\phi}(\xi)=0$ for almost every $\xi\notin G$. 
Consequently any $f\in V$ is in $H^2_G$, so $V\subseteq H^2_G$. 
This proves (iv).

For (vi) firstly note that \eqref{eq:support of spectral function} is obvious if 
$\mathcal{F}$ is a tight frame generator of $V$. 
Then it is enough to remark that if $f$ is in the linear span of the shifts of elements in $\mathcal{F}$, 
then $\widehat{f}$ is a sum of products of trigonometric polynomials times the Fourier transforms of these elements. 
Consequently $Supp\widehat{f}$ is included in the right side of \eqref{eq:support of spectral function}. 
This easily implies that for any $\mathcal{F}$ this set coincides.
}
\end{proof}
One can see that $\{\phi_{(j,k)}^{\alpha}:=d_A^{\frac{j}{2}}\,\phi^{\alpha}(A^j\cdot-k)\::\:k\in\mathbb{Z}^d,\alpha\in I\}$ is a tight frame of $\mathcal{D}_{A}^{j}V$ for all $j\in\mathbb{Z}$ (see \cite{BR}). 
Furthermore, if we denote by $P_j$ the orthogonal projection operator on $\mathcal{D}_{A}^{j}V$ for any integer $j$, we have
\begin{equation} \label{eq:formula projections}
P_jf=
\sum_{\alpha\in I}\sum_{k\in\mathbb{Z}^d}<f,\phi_{(j,k)}^{\alpha}>\,\phi_{(j,k)}^{\alpha}
\quad\mbox{for all }f\in L^2(\mathbb{R}^{d})\,.
\end{equation}


In \cite{KSa} some generalizations of Definitions \ref{defi:point of A-density} and \ref{defi:point of A-approximate continuity} were introduced. 
They allow us to deal with the general case in $A$-reducing spaces. 
Fix an $A$-invariant set $G$.
\begin{definition}
Let $E\subseteq\mathbb{R}^d$ be a set of positive measure. 
We say that the origin is a point of $(G,A)$-density for $E$ if for all $r>0$
\begin{displaymath}
\lim_{j\to\infty}\frac{|E\cap G\cap A^{-j}B_r|}{|G\cap A^{-j}B_r|}=1\,.
\end{displaymath}
\end{definition}
It is easy to prove that if $F\subseteq\mathbb{R}^d$ measurable is such that 
there exist $r_1,r_2\in\mathbb{R}$ with $0<r_1<r_2<\infty$ and $B_{r_1}\subseteq F\subseteq B_{r_2}$, 
then the origin is a point of $(G,A)$-density for $E$ if and only if
\begin{displaymath}
\lim_{j\to\infty}\frac{|E\cap G\cap A^{-j}F|}{|G\cap A^{-j}F|}=1\,,
\end{displaymath}
or equivalently
\begin{displaymath}
\lim_{j\to\infty}\frac{|(\mathbb{R}^d\setminus E)\cap G\cap A^{-j}F|}{|G\cap A^{-j}F|}=0\,.
\end{displaymath}
Note that if we put $G=\mathbb{R}^d$ in the above definition we recover Definition \ref{defi:point of A-density}: 
that is, the origin is a point of $A$-density for $E$.
\begin{definition}
Let $f\,:\,\mathbb{R}^d\longrightarrow\mathbb{C}$ be a measurable function. 
We say that the origin is a point of $(G,A)$-approximate continuity of $f$, 
if there exists $E\subseteq\mathbb{R}^d$ of positive measure such that the origin is a point of $(G,A)$-density for $E$ and
\begin{displaymath}
\lim_{\begin{subarray}{c} x\to 0 \\ x\in E \end{subarray}}f(x)=f(0)\,.
\end{displaymath}
We also say that $f$ is $(G,A)$-locally nonzero at the origin if
\begin{displaymath}
\lim_{j\to\infty}\frac{|\{x\in G\cap A^{-j}B_1\::\:f(x)=0\}|}{|G\cap A^{-j}B_1|}=0\,.
\end{displaymath}
That is, if the origin is a point of $(G,A)$-density for $Supp(f)$.
\end{definition}
Similarly, if we put $G=\mathbb{R}^d$ in this definition then we will take that 
the origin is a point of $A$-approximate continuity of $f$, 
and $f$ is $A$-locally nonzero at the origin, respectively.

The following characterizations of approximate continuity hold:
\begin{lemma} \label{lemma:conditions for approximate continuity}
Let $f\,:\,\mathbb{R}^d\longrightarrow\mathbb{C}$ be a measurable function, 
and $F\subseteq\mathbb{R}^d$ measurable such that there exist $r_1,r_2\in\mathbb{R}$ with $0<r_1<r_2<\infty$ and $B_{r_1}\subseteq F\subseteq B_{r_2}$. 
Then the following conditions are equivalent:
\begin{description}
	\item[(i)] the origin is a point of $(G,A)$-approximate continuity of $f$;
	\item[(ii)] for all $\varepsilon>0$
	\begin{displaymath}
	\lim_{j\to\infty}\frac{|\{x\in G\cap A^{-j}F\::\:|f(x)-f(0)|<\varepsilon\}|}{|G\cap A^{-j}F|}=1\,;
	\end{displaymath}
	\item[(iii)] for all $\varepsilon>0$ there exists a positive integer $j_0$ such that for all $j\geq j_0$
	\begin{displaymath}
	\frac{|\{x\in G\cap A^{-j}F\::\:|f(x)-f(0)|\geq\varepsilon\}|}{|G\cap A^{-j}F|}<\varepsilon\,.
	\end{displaymath}
\end{description}
\end{lemma}
\begin{proof}
{(i)$\Longrightarrow$(ii) 
We suppose that the origin is a point of $(G,A)$-approximate continuity of $f$, 
so there exists a measurable set $E\subseteq\mathbb{R}^d$ such that
\begin{displaymath}
\lim_{j\to\infty}\frac{|E\cap G\cap A^{-j}F|}{|G\cap A^{-j}F|}=1
\quad\mbox{and}\quad
\lim_{\begin{subarray}{c} x\to 0 \\ x\in E \end{subarray}}f(x)=f(0)\,.
\end{displaymath}
Thus, if we fix $\varepsilon>0$ and $\widetilde{\varepsilon}>0$, 
there exist $\delta>0$ and a positive integer $j_1$ such that for all $j\geq j_1$ and $x\in B_{\delta}\cap E$
\begin{displaymath}
1-\frac{|E\cap G\cap A^{-j}F|}{|G\cap A^{-j}F|}<\widetilde{\varepsilon}
\quad\mbox{and}\quad
|f(x)-f(0)|<\varepsilon\,.
\end{displaymath}
Since $A$ is an expansive dilation, 
there exists a positive integer $j_2$ such that for all $j\geq j_2$ one has $A^{-j}B_{r_2}\subseteq B_{\delta}$, 
so $G\cap A^{-j}F\subseteq A^{-j}B_{r_2}\subseteq B_{\delta}$. 
Consequently, if $x\in E\cap G\cap A^{-j}F$ one has $|f(x)-f(0)|<\varepsilon$. 
This implies $E\cap G\cap A^{-j}F\subseteq\{x\in G\cap A^{-j}F\::\:|f(x)-f(0)|<\varepsilon\}$. 
Finally, if $j\geq j_0:=\max\{j_1,j_2\}$ one has
\begin{displaymath}
1-\frac{|\{x\in G\cap A^{-j}F\::\:|f(x)-f(0)|<\varepsilon\}|}{|G\cap A^{-j}F|}
\leq 1-\frac{|E\cap G\cap A^{-j}F|}{|G\cap A^{-j}F|}
<\widetilde{\varepsilon}\,.
\end{displaymath}
As $\widetilde{\varepsilon}$ is arbitrary, one gets $(ii)$.

(ii)$\Longrightarrow$(iii) 
Fix $\varepsilon>0,\,\widetilde{\varepsilon}>0$. 
From (ii) we know that there exists a positive integer $j_0$ such that for all $j\geq j_0$
\begin{displaymath}
1-\frac{|\{x\in G\cap A^{-j}F\::\:|f(x)-f(0)|<\varepsilon\}|}{|G\cap A^{-j}F|}
<\widetilde{\varepsilon}\,.
\end{displaymath}
Then
\begin{displaymath}
\frac{|\{x\in G\cap A^{-j}F\::\:|f(x)-f(0)|\geq\varepsilon\}|}{|G\cap A^{-j}F|}=
\end{displaymath}
\begin{displaymath}
=\frac{|G\cap A^{-j}F|}{|G\cap A^{-j}F|}-\frac{|\{x\in G\cap A^{-j}F\::\:|f(x)-f(0)|<\varepsilon\}|}{|G\cap A^{-j}F|}
<\widetilde{\varepsilon}\,.
\end{displaymath}
Choosing $\widetilde{\varepsilon}=\varepsilon$ one gets (iii).

(iii)$\Longrightarrow$(i) 
Suppose (iii), and for any positive integer $n$ make $\varepsilon:=2^{-n-1}>0$. 
Then there exists a positive integer $j_n$ such that for all $j\geq j_n$
\begin{displaymath}
\frac{|\{x\in G\cap A^{-j}F\::\:|f(x)-f(0)|\geq 2^{-n-1}\}|}{|G\cap A^{-j}F|}
<\frac{1}{2^{n+1}}\,.
\end{displaymath}
Note that $\{j_n\}_{n=1}^{\infty}$ can be taken strictly increasing. 
Define for any $n$
\begin{displaymath}
F_n:=\{x\in G\cap A^{-j_n}F\::\:|f(x)-f(0)|\geq 2^{-n-1}\}\,.
\end{displaymath}
We recall we have $\frac{|F_n|}{|G\cap A^{-j_n}F|}<2^{-n-1}$. 
Set now
\begin{displaymath}
E:=(G\cap F)\setminus\bigcup_{n=1}^{\infty}F_n\,.
\end{displaymath}
The set $E$ is evidently measurable. 
It is not empty also, as
\begin{displaymath}
\frac{|E|}{|G\cap F|}\geq 
1-\frac{|\bigcup_{n=1}^{\infty}F_n|}{|G\cap F|}\geq 
1-\sum_{n=1}^{\infty}\frac{|F_n|}{|G\cap F|}\geq
\end{displaymath}
\begin{displaymath}
\geq 1-\sum_{n=1}^{\infty}\frac{|F_n|}{d_A^{-j_n}\,|G\cap F|}
=1-\sum_{n=1}^{\infty}\frac{|F_n|}{|G\cap A^{-j_n}F|}
>1-\sum_{n=1}^{\infty}\frac{1}{2^{n+1}}
=\frac{1}{2}\,.
\end{displaymath}

Since $A$ is an expansive dilation, 
there exists a positive integer $j_0$ such that $A^{-j}F\subseteq F$ if $j\geq j_0$. 
Fixed $j\geq\max\{j_0,j_1\}$, let $n$ be a positive integer satisfying $j_n\leq j<j_{n+1}$. 
One has $\bigcup_{m=1}^nF_m\subseteq\{x\::\:|f(x)-f(0)|\geq 2^{-n-1}\}$, so
\begin{displaymath}
\frac{|(\mathbb{R}^d\setminus E)\cap G\cap A^{-j}F|}{|G\cap A^{-j}F|}\leq
\frac{|(\bigcup_{m=1}^{\infty}F_m)\cap G\cap A^{-j}F|}{|G\cap A^{-j}F|}\leq
\end{displaymath}
\begin{displaymath}
\leq\frac{|(\bigcup_{m=1}^{n}F_m)\cap G\cap A^{-j}F|}{|G\cap A^{-j}F|}+
\frac{|(\bigcup_{m=n+1}^{\infty}F_m)\cap G\cap A^{-j}F|}{|G\cap A^{-j}F|}\leq
\end{displaymath}
\begin{displaymath}
\leq\frac{|\{x\in G\cap A^{-j}F\::\:|f(x)-f(0)|\geq 2^{-n-1}\}|}{|G\cap A^{-j}F|}
+\sum_{m=n+1}^{\infty}\frac{|F_m\cap G\cap A^{-j}F|}{|G\cap A^{-j}F|}<
\end{displaymath}
\begin{displaymath}
<\frac{1}{2^{n+1}}+\sum_{m=n+1}^{\infty}\frac{|F_m|}{|G\cap A^{-j_m}F|}
<\frac{1}{2^{n+1}}+\sum_{m=n+1}^{\infty}\frac{1}{2^{m+1}}
=\frac{1}{2^{n}}\xrightarrow[j\to\infty]{}0\,.
\end{displaymath}
This is equivalent to the fact that the origin is point of $(G,A)$-density for $E$.

Finally let $n_0$ be a positive integer such that $2^{-n_0-1}<\varepsilon$, 
and $\delta>0$ such that $B_{\delta}\subseteq A^{-j_{n_0}}B_{r_1}$, 
so $B_{\delta}\subseteq A^{-j_{n_0}}B_{r_1}\subseteq A^{-j_{n_0}}F$. 
Thus, if $x\in E\cap B_{\delta}$ then $x\in G\cap A^{-j_{n_0}}F$ and $x\notin F_{n_0}$, so $|f(x)-f(0)|<2^{-n_0-1}<\varepsilon$. 
This proves that $\lim_{\begin{subarray}{c} x\to 0 \\ x\in E \end{subarray}}f(x)=f(0)$, 
so the origin is a point of $(G,A)$-approximate continuity of $f$.}
\end{proof}

The following lemma is proved in the same way as the corresponding result in \cite{Sa2}.
\begin{lemma} \label{lemma:sequences and approximate continuity}
Let $f\,:\,\mathbb{R}^d\longrightarrow\mathbb{C}$ be a measurable function. 
If one has
\begin{displaymath}
\lim_{j\to\infty}f(A^{-j}x)=f(0)\quad\mbox{for }a.e.\,x\in G
\end{displaymath}
then the origin is a point of $(G,A)$-approximate continuity of $f$. 
On the other hand, if the origin is a point of $(G,A)$-approximate continuity of $f$ 
then there exists an increasing sequence of positive integers $\{j_k\}_{k=1}^{\infty}$ such that
\begin{displaymath}
\lim_{k\to\infty}f(A^{-j_k}x)=f(0)\quad\mbox{for }a.e.\,x\in G\,.
\end{displaymath}
\end{lemma}
\begin{proof}
{Suppose $\lim_{j\to\infty}f(A^{-j}x)=f(0)$ for almost every $x\in G$ and pick $\varepsilon>0$. 
We define for any positive integer $j$
\begin{displaymath}
F_j:=\{x\in G\cap B_1\::\:|f(A^{-j}x)-f(0)|<\varepsilon\}
\end{displaymath}
and
\begin{displaymath}
E_j:=\bigcap_{k\geq j}F_k=
\{x\in G\cap B_1\::\:|f(A^{-k}x)-f(0)|<\varepsilon\:\:\mbox{for all}\:k\geq j\}\,.
\end{displaymath}
By definition $\{E_j\}_{j=1}^{\infty}$ is increasing and $\liminf_{j\to\infty}F_j=\bigcup_{j=1}^{\infty}E_j=\lim_{j\to\infty}E_j$. 
From our hypothesis we see that for almost every $x\in G\cap B_1$ there exists a positive integer $j_0$ 
such that for all $j\geq j_0$ $|f(A^{-j}x)-f(0)|<\varepsilon$. 
Thus $|G\cap B_1\setminus\bigcup_{j=1}^{\infty}E_j|=0$, 
so $\lim_{j\to\infty}|E_j|=|\bigcup_{j=1}^{\infty}E_j|=|G\cap B_1|$ and finally
\begin{displaymath}
1\geq
\frac{|\{x\in G\cap A^{-j}B_1\::\:|f(x)-f(0)|<\varepsilon\}|}{|G\cap A^{-j}B_1|}=
\frac{|F_j|}{|G\cap B_1|}\geq
\frac{|E_j|}{|G\cap B_1|}\xrightarrow[j\to\infty]{}1\,.
\end{displaymath}
By Lemma \ref{lemma:conditions for approximate continuity} the origin is a point of $(G,A)$-approximate continuity of $f$.

On the other hand, suppose that the origin is a point of $(G,A)$-approximate continuity of $f$. 
By Lemma \ref{lemma:conditions for approximate continuity} for any $r>0,\varepsilon>0$ one has
\begin{displaymath}
\frac{|\{x\in G\cap A^{-j}B_r\::\:|f(x)-f(0)|<\varepsilon\}|}{|G\cap A^{-j}B_r|}=
\end{displaymath}
\begin{displaymath}
=\frac{|\{x\in G\cap B_r\::\:|f(A^{-j}x)-f(0)|<\varepsilon\}|}{|G\cap B_r|}
\xrightarrow[j\to\infty]{}1\,,
\end{displaymath}
so $|\{x\in G\cap B_r\::\:|f(A^{-j}x)-f(0)|\geq\varepsilon\}|\xrightarrow[j\to\infty]{}0$. 
This is equivalent to the fact that $f(A^{-j}\cdot)\xrightarrow[j\to\infty]{}f(0)$ in measure on $G\cap B_r$ for all $r>0$.

Applying the Riesz theorem for any positive integer $r$ 
we can find an increasing subsequence $\{j_k^r\}_{k=1}^{\infty}$ of positive integers in such a way that $\{j_k^r\}_{k=1}^{\infty}\subseteq\{j_k^{r-1}\}_{k=1}^{\infty}$ for all $r$ 
and such that $f(A^{-j_k^r}x)\xrightarrow[k\to\infty]{}f(0)$ for almost every $x\in G\cap B_r$. 
Using Cantor's diagonal method of election we obtain that $f(A^{-j_k^k}x)\xrightarrow[k\to\infty]{}f(0)$ for almost every $x\in G$.}
\end{proof}

\section{Completeness property} \label{sec:completeness}

We present now several characterizations of the completeness property of $A$-refinable shift-invariant subspaces. 
They generalize Theorem \ref{thm:characterization Dutkay} and all the above-cited results on characterization of (one or several) scaling functions. 
As a corollary one obtains Theorem \ref{thm:characterizations for L2}, simply considering the case $G=\mathbb{R}^d$.

\begin{theorem} \label{thm:characterization one}
Let $V$ be an $A$-refinable shift-invariant subspace and assume that $\mathcal{G}$ generates $V$. 
Then $\overline{\bigcup_{j\in\mathbb{Z}}\mathcal{D}_{A}^{j}V}=H^2_G$, where
\begin{equation} \label{eq:G in terms of supports}
G:=\bigcup_{j\in\mathbb{Z}}A^{*j}(Supp(\sigma_V))
=\bigcup_{\phi\in\mathcal{G}}\bigcup_{j\in\mathbb{Z}}A^{*j}(Supp(\widehat{\phi}))\,.
\end{equation}
\end{theorem}
\begin{proof}
{Firstly note that $G$ is obviously an $A^*$-invariant set, 
and \eqref{eq:G in terms of supports} follows from \eqref{eq:support of spectral function}. 
Clearly $\overline{\bigcup_{j\in\mathbb{Z}}\mathcal{D}_{A}^{j}V}$ is $A$-reducing, 
so equals $H^2_{\widetilde{G}}$ for some $A^*$-invariant set $\widetilde{G}$. 
Evidently $G\subseteq\widetilde{G}$. 
Let $P_j$ be the orthogonal projection operator on $\mathcal{D}_{A}^{j}V$ for any integer $j$. 
It is easy to see that $P_jf\xrightarrow[j\to+\infty]{}f$ in $L^2(\mathbb{R}^{d})$ for all $f\in H^2_{\widetilde{G}}$. 
We note that for any $j\in\mathbb{Z}$ and $f\in H^2_{\widetilde{G}}$ 
$P_jf$ is supported on $G$ (see formula \eqref{eq:formula projections}). 
Hence we deduce from this that $\widetilde{G}\subseteq G$, and we get that $G=\widetilde{G}$.}
\end{proof}

Now we present another characterization, but before that we introduce a new notion. 
Following Definition \ref{defi:A-absorbing}, we say that a measurable set $E\subseteq\mathbb{R}^d$ is $A$-absorbing in $G$ 
if for almost every $\xi\in G$ there exists a positive integer $j_0$ (possibly dependent on $\xi$) such that $A^{j}\xi\in E$ if $j\geq j_0$.

The following comment will be of importance in what follows. 
From the $A$-refinability of $V$ and Proposition \ref{prop:properties spectral function} we deduce that
\begin{displaymath}
\sigma_V(A^{*-1}\xi)=\sigma_{\mathcal{D}_AV}(\xi)\geq\sigma_V(\xi)
\quad\mbox{for }a.e.\,\xi\in\mathbb{R}^d\,.
\end{displaymath}
Hence, the sequences $\{\sigma_V(A^{*-j}\xi)\}_{j=1}^{\infty}$ are non decreasing, 
and consequently convergent for almost every $\xi\in\mathbb{R}^d$ 
(recall that the spectral function is bounded above by $1$). 
Note that in the so called principal case 
(the tight frame generator has only one member $\phi$) 
this monotony property follows from the scaling equation \eqref{eq:scaling equation}, 
since the low-pass filter $m$ satisfies $|m(\xi)|\leq 1$ for almost every $\xi\in\mathbb{R}^d$.

Applying this fact and Theorem \ref{thm:characterization one} thus we are able to prove the following:

\begin{theorem} \label{thm:characterization two}
Let $V\neq\{0\}$ be an $A$-refinable shift-invariant subspace and 
$G\subseteq\mathbb{R}^d$ an $A^*$-invariant set of positive measure such that 
$Supp(\sigma_V)\subseteq G$. 
Then the following conditions are equivalent:
\begin{description}
	\item[(a)] $\overline{\bigcup_{j\in\mathbb{Z}}\mathcal{D}_{A}^{j}V}=H^2_G$;
	\item[(b)] the set $Supp(\sigma_V)$ is $A^{*-1}$-absorbing in $G$;
	\item[(c)] $\lim_{j\to\infty}\sigma_V(A^{*-j}\xi)>0$ for almost every $\xi\in G$.
\end{description}
\end{theorem}
\begin{proof}
{First of all, by Theorem \ref{thm:characterization one} we have $\overline{\bigcup_{j\in\mathbb{Z}}\mathcal{D}_{A}^{j}V}=H^2_{\widetilde{G}}$ 
for the $A^*$-set
\begin{displaymath}
\widetilde{G}:=\bigcup_{j\in\mathbb{Z}}A^{*j}(Supp(\sigma_V))\,.
\end{displaymath}
As $Supp(\sigma_V)\subseteq G$ and $G$ is $A^*$-invariant, 
we clearly have $\widetilde{G}\subseteq G$.

If we start from (b), 
for almost every $\xi\in G$ there exists an integer $j_0$ such that $\sigma_V(A^{*-j}\xi)>0$ for $j\geq j_0$. 
Then $\xi=A^{*j_0}A^{*-j_0}\xi\in A^{*j_0}(Supp(\sigma_V))\subseteq\widetilde{G}$. 
This proves $G\subseteq\widetilde{G}$ and also (a).

Analogously, from (c) 
we know that for almost every $\xi\in G$ there exists an integer $j$ such that $\sigma_V(A^{*-j}\xi)>0$. 
Then $\xi=A^{*j}A^{*-j}\xi\in A^{*j}(Supp(\sigma_V))\subseteq\widetilde{G}$. 
This proves $G\subseteq\widetilde{G}$ and also (a).

Reciprocally, if we suppose (a) 
by Theorem \ref{thm:characterization one} we have \eqref{eq:G in terms of supports}. 
Thus, for almost every $\xi\in G$ there exist $\eta\in Supp(\sigma_V)$ and an integer $j_0$ such that $\xi=A^{*j_0}\eta$. 
From the comment above we deduce that $\sigma_V(A^{*-j}\xi)=\sigma_V(A^{*j_0-j}\eta)\geq\sigma_V(\eta)>0$ for all $j\geq j_0$, 
so $\lim_{j\to\infty}\sigma_V(A^{*-j}\xi)\geq\sigma_V(\eta)>0$. 
We have shown (b) and (c).}
\end{proof}

Note that condition $Supp(\sigma_V)\subseteq G$ is a very natural one. 
If we suppose $|Supp(\sigma_V)\setminus G|>0$ then the result is obviously false. 
The following characterization gives some additional information about the behaviour of the spectral function at the origin. 
In fact it expresses some kind of normalization at the origin.

\begin{theorem} \label{thm:characterization three}
Let $V\neq\{0\}$ be an $A$-refinable shift-invariant subspace and 
$G\subseteq\mathbb{R}^d$ an $A^*$-invariant set of positive measure such that $Supp(\sigma_V)\subseteq G$. 
Then the following conditions are equivalent:
\begin{description}
	\item[(a)] $\overline{\bigcup_{j\in\mathbb{Z}}\mathcal{D}_{A}^{j}V}=H^2_G\,$;
	\item[(b)] for any bounded $E\subseteq G$ of positive measure one has
	\begin{displaymath}
	\lim_{j\to\infty}\frac{1}{|A^{*-j}E|}\,\int_{A^{*-j}E}\sigma_V(\xi)\,d\xi=1\,.
	\end{displaymath}
\end{description}
\end{theorem}
\begin{proof}
{Suppose (a) and pick $E$ as in (b). 
Define $f\in H^2_G$ such that $\widehat{f}:={\text{\Large $\chi$}}_E$ and 
let $\mathcal{G}=\{\phi^{\alpha}\}_{\alpha\in I}$ be a tight frame generator of $V$. 
Since $A^{*}$ is an expansive dilation and $E$ is bounded, 
there exists a positive integer $l$ such that for all $j\geq l$ $A^{*-j}(E)\subseteq [-\frac{1}{2},\frac{1}{2}]^d$, 
and using the fact that $\{\phi_{(j,k)}^{\alpha}:=d_A^{\frac{j}{2}}\,\phi^{\alpha}(A^j\cdot-k)\::\:\alpha\in I,k\in\mathbb{Z}^d\}$ is a tight frame of $\mathcal{D}_{A}^{j}V$ 
by \eqref{eq:formula projections} we have
\begin{displaymath}
\left\|\,P_jf\,\right\|^{2}=
\sum_{\alpha\in I}\sum_{k\in\mathbb{Z}^d}|<f,\phi_{(j,k)}^{\alpha}>|^2=
\sum_{\alpha\in I}\sum_{k\in\mathbb{Z}^d}|<\widehat{f},\widehat{\phi_{(j,k)}^{\alpha}}>|^2=
\end{displaymath}
\begin{displaymath}
=\sum_{\alpha\in I}\sum_{k\in\mathbb{Z}^d}\left|d_A^{-\frac{j}{2}}\,\int_{\mathbb{R}^d}\widehat{f}(\xi)\,e^{2\pi ik\cdot A^{*-j}\xi}\,\overline{\widehat{\phi^{\alpha}}(A^{*-j}\xi)}\,d\xi\right|^2=
\end{displaymath}
\begin{displaymath}
=\sum_{\alpha\in I}\sum_{k\in\mathbb{Z}^d}\left|d_A^{\frac{j}{2}}\,\int_{\mathbb{R}^d}\widehat{f}(A^{*j}\eta)\,e^{2\pi ik\cdot\eta}\,\overline{\widehat{\phi^{\alpha}}(\eta)}\,d\eta\right|^2=
\end{displaymath}
\begin{displaymath}
=d_A^{j}\,\sum_{\alpha\in I}\sum_{k\in\mathbb{Z}^d}\left|\int_{A^{*-j}E}\overline{\widehat{\phi^{\alpha}}(\eta)}\,e^{2\pi ik\cdot\eta}\,d\eta\right|^2=
\end{displaymath}
\begin{displaymath}
=d_A^{j}\,\sum_{\alpha\in I}\int_{A^{*-j}E}|\widehat{\phi^{\alpha}}(\eta)|^2\,d\eta
=d_A^{j}\,\int_{A^{*-j}E}\sigma_V(\eta)\,d\eta\,.
\end{displaymath}
Since
\begin{displaymath}
\left\|\,P_jf\,\right\|^{2}\xrightarrow[j\to+\infty]{}\left\|\,f\,\right\|^{2}
=|E|\,,
\end{displaymath}
we have
\begin{displaymath}
\lim_{j\to\infty}\frac{1}{|A^{*-j}E|}\,\int_{A^{*-j}E}\sigma_V(\xi)\,d\xi=
\lim_{j\to\infty}\frac{d_A^j}{|E|}\,\int_{A^{*-j}E}\sigma_V(\xi)\,d\xi=1\,.
\end{displaymath}

To prove the reciprocal, we use Theorem \ref{thm:characterization one}. 
We obtain $\overline{\bigcup_{j\in\mathbb{Z}}\mathcal{D}_{A}^{j}V}=H^2_{\widetilde{G}}$, 
where $\widetilde{G}$ is the $A^*$-invariant set defined by \eqref{eq:G in terms of supports}. 
From our hypothesis we trivially deduce $\widetilde{G}\subseteq G$. 
We proceed by \emph{reductio ad absurdum} and we suppose $|G\setminus\widetilde{G}|>0$. 
Note that $G\setminus\widetilde{G}$ is also an $A^*$-invariant set, 
so we know that $|(G\setminus\widetilde{G})\cap B_r|>0$ for all $r>0$ 
(see \cite{KSa}). 
Since $V\subseteq H^2_{\widetilde{G}}$ we get $\sigma_V\leq{\text{\Large $\chi$}}_{\widetilde{G}}$. 
Applying (b) one has
\begin{displaymath}
1=\lim_{j\to\infty}\frac{1}{|A^{*-j}((G\setminus\widetilde{G})\cap B_r)|}\,\int_{A^{*-j}((G\setminus\widetilde{G})\cap B_r)}\sigma_V(\xi)\,d\xi\leq
\end{displaymath}
\begin{displaymath}
\leq\lim_{j\to\infty}\frac{|\widetilde{G}\cap(G\setminus\widetilde{G})\cap A^{*-j}B_r|}{|(G\setminus\widetilde{G})\cap A^{*-j}B_r|}=0\,.
\end{displaymath}
This proves $G=\widetilde{G}$ and (a).}
\end{proof}

Now we present another characterization using the above introduced notions.
\begin{theorem} \label{thm:characterization four}
Let $V\neq\{0\}$ be an $A$-refinable shift-invariant subspace and 
$G\subseteq\mathbb{R}^d$ an $A^*$-invariant set of positive measure such that  $Supp(\sigma_V)\subseteq G$. 
Then the following conditions are equivalent:
\begin{description}
	\item[(a)] $\overline{\bigcup_{j\in\mathbb{Z}}\mathcal{D}_{A}^{j}V}=H^2_G\,$;
	\item[(b)] $\sigma_V$ is $(G,A^*)$-locally nonzero at the origin;
	\item[(c)] the origin is a point of $(G,A^*)$-approximate continuity of $\sigma_V$ if we put $\sigma_V(0)=1$.
\end{description}
\end{theorem}
\begin{proof}
{(b)$\Rightarrow$(a) 
By Theorem \ref{thm:characterization one} $\overline{\bigcup_{j\in\mathbb{Z}}\mathcal{D}_{A}^{j}V}=H^2_{\widetilde{G}}$ for some $A^*$-invariant set $\widetilde{G}$, 
and from our hypothesis we have $\widetilde{G}\subseteq G$. 
For any nonnegative integer $j$ $\sigma_{\mathcal{D}_{A}^{j}V}\leq{\text{\Large $\chi$}}_{\widetilde{G}}$, 
so $\sigma_{\mathcal{D}_{A}^{j}V}(\xi)\,\widehat{g}(\xi)=\sigma_{V}(A^{*-j}\xi)\,\widehat{g}(\xi)=0$ 
for any $g\in(H^2_{\widetilde{G}})^{\perp}=H^2_{\mathbb{R}^d\setminus\widetilde{G}}$ and almost every $\xi\in\mathbb{R}^d$, 
or equivalently
\begin{displaymath}
\sigma_{V}(\xi)\,\widehat{g}(A^{*j}\xi)=0
\quad a.e.\,\xi\in\mathbb{R}^d,\forall\:j\geq 0\,.
\end{displaymath}

Let $r,\varepsilon>0$. 
Since $\sigma_V$ is $(G,A^*)$-locally nonzero at the origin, 
there exists a positive integer $j$, dependent on $r$ and $\varepsilon$ such that
\begin{displaymath}
\frac{|\{x\in A^{*-j}B_r\cap G\::\:\sigma_V(x)=0\}|}{|A^{*-j}B_r\cap G|}
<\varepsilon\,.
\end{displaymath}
Then,
\begin{displaymath}
\frac{|\{y\in B_r\cap G\::\:\widehat{g}(y)\neq 0\}|}{|B_r\cap G|}=
\frac{|\{x\in A^{*-j}(B_r\cap G)\::\:\widehat{g}(A^{*j}x)\neq 0\}|}{|A^{*-j}(B_r\cap G)|}\leq
\end{displaymath}
\begin{displaymath}
\leq\frac{|\{x\in A^{*-j}(B_r\cap G)\::\:\sigma_V(x)=0\}|}{|A^{*-j}(B_r\cap G)|}
<\varepsilon\,.
\end{displaymath}
As $\varepsilon>0$ and $r>0$ are arbitrary, 
we get $H^2_{\mathbb{R}^d\setminus\widetilde{G}}\subseteq H^2_{\mathbb{R}^d\setminus G}$, 
so $G=\widetilde{G}$. 
This proves $(a)$.

(a)$\Rightarrow$(c) 
We will prove that for all $\varepsilon,r>0$ one has
\begin{displaymath}
\lim_{j\to\infty}\frac{|\{x\in A^{*-j}(B_r\cap G)\::\:1-\sigma_V(x)<\varepsilon\}|}{|A^{*-j}(B_r\cap G)|}=1\,.
\end{displaymath}
If this is not true, 
there exist $\varepsilon_0,0<\varepsilon_0<1$, $r_0>0$ and $\{j_n\}_{n=1}^{\infty},j_n\nearrow\infty$ 
such that for any positive integer $n$
\begin{displaymath}
|E_n|\geq\varepsilon_0\,|A^{*-j_n}(B_{r_0}\cap G)|,
\quad\mbox{where}\quad
E_n:=\{x\in A^{*-j_n}(B_{r_0}\cap G)\::\:1-\sigma_V(x)>\varepsilon_0\}\,.
\end{displaymath}
By Theorem \ref{thm:characterization three}
\begin{displaymath}
1=\lim_{n\to\infty}\frac{1}{|A^{*-j_n}(B_{r_0}\cap G)|}\,\int_{A^{*-j_n}(B_{r_0}\cap G)}\sigma_V(x)\,dx=
\end{displaymath}
\begin{displaymath}
=\lim_{n\to\infty}\frac{1}{|A^{*-j_n}(B_{r_0}\cap G)|}\,\left(\int_{A^{*-j_n}(B_{r_0}\cap G)\setminus E_n}\sigma_V(x)\,dx+\int_{E_n}\sigma_V(x)\,dx\right)<
\end{displaymath}
\begin{displaymath}
<\lim_{n\to\infty}\frac{|A^{*-j_n}(B_{r_0}\cap G)|-|E_n|+(1-\varepsilon_0)\,|E_n|}{|A^{*-j_n}(B_{r_0}\cap G)|}
=\lim_{n\to\infty}\frac{|A^{*-j_n}(B_{r_0}\cap G)|-\varepsilon_0\,|E_n|}{|A^{*-j_n}(B_{r_0}\cap G)|}\leq
\end{displaymath}
\begin{displaymath}
\leq\lim_{n\to\infty}\frac{(1-\varepsilon_0^2)\,|A^{*-j_n}(B_{r_0}\cap G)|}{|A^{*-j_n}(B_{r_0}\cap G)|}
=1-\varepsilon_0^2<1\,.
\end{displaymath}
By Lemma \ref{lemma:conditions for approximate continuity} we have (c).

(c)$\Rightarrow$(b) 
is easy to see, since we put $\sigma_V(0)=1\neq 0$.}
\end{proof}

The following characterization generalizes Theorem \ref{thm:characterization Dutkay}.
\begin{theorem} \label{thm:characterization five}
Let $V\neq\{0\}$ be an $A$-refinable shift-invariant subspace and 
$G\subseteq\mathbb{R}^d$ an $A^*$-invariant set of positive measure such that  $Supp(\sigma_V)\subseteq G$. 
Then the following conditions are equivalent:
\begin{description}
	\item[(a)] $\overline{\bigcup_{j\in\mathbb{Z}}\mathcal{D}_A^{j}V}=H^2_G\,$;
	\item[(b)] $\lim_{j\to\infty}\sigma_V(A^{*-j}\xi)=1$ for almost every $\xi\in G$.
\end{description}
\end{theorem}
\begin{proof}
{We have already observed that the limit $\lim_{j\to\infty}\sigma_V(A^{*-j}\xi)$ exists for almost every $\xi\in\mathbb{R}^d$, and is nonnegative. 
Also, by Theorem \ref{thm:characterization four} 
condition (a) is equivalent to the fact that the origin is a point of $(G,A^*)$-approximate continuity of $\sigma_V$ if we put $\sigma_V(0)=1$.

If we suppose (a), 
by Lemma \ref{lemma:sequences and approximate continuity} there exists an increasing sequence of positive integers $\{j_k\}_{k=1}^{\infty}$ 
such that $\lim_{k\to\infty}\sigma_V(A^{*-j_k}\xi)=1$ for almost every $\xi\in G$. 
By using the monotonicity of the sequences $\{\sigma_V(A^{*-j}\xi)\}_{j=1}^{\infty}$ we arrive to (b). 
Reciprocally, by Lemma \ref{lemma:sequences and approximate continuity} 
(b) implies that the origin is a point of $(G,A^*)$-approximate continuity of $\sigma_V$ if we put $\sigma_V(0)=1$.}
\end{proof}

Obviously the given characterizations depend on the map $A$. 
Actually the $A$-refinability also do. 
In \cite{RSa} it is studied when a pair of self-adjoint expansive dilations give rise to equivalent conditions of $A$-approximate continuity.

\section{Wavelets} \label{sec:wavelets}

The generalized multiresolution analyses and multiwavelets are very close objects. 
We exploit this closeness to make some comments on the local behaviour of the Fourier transform of the multiwavelets. 
From Theorem \ref{thm:characterization four} we obtain the following result:
\begin{proposition} \label{prop:main consequence}
Let a shift-invariant subspace $W$ such that for some expansive dilation $A$ and $A^*$-invariant set $G$ we have
\begin{equation} \label{eq:orthogonal decomposition}
H^2_G=\bigoplus_{j\in\mathbb{Z}}\mathcal{D}_{A}^{j}W\,.
\end{equation}
Then the origin is a point of $A^*$-approximate continuity of $\sigma_W$ if we put $\sigma_W(0)=0$.
\end{proposition}
\begin{proof}
{If we define
\begin{displaymath}
V=\bigoplus_{j<0}\mathcal{D}_{A}^{j}W
\end{displaymath}
we have that the subspace $V^{\perp}=H^2_{\mathbb{R}^d\setminus G}\oplus(\bigoplus_{j\geq 0}\mathcal{D}_{A}^{j}W)$ is shift-invariant, 
so $V$ is also. 
Clearly $Supp(\sigma_V)\subseteq G$ and $V$ is also $A$-refinable, 
since for any $l\in\mathbb{Z}$ $\mathcal{D}_{A}^{l}V=\bigoplus_{j<l}\mathcal{D}_{A}^{j}W$. 
We have $\mathcal{D}_{A}V=V\oplus W$, 
so by Proposition \ref{prop:properties spectral function}
\begin{equation} \label{eq:spectral function W}
\sigma_W(\xi)=
\sigma_{\mathcal{D}_{A}V}(\xi)-\sigma_V(\xi)=
\sigma_V(A^{*-1}\xi)-\sigma_V(\xi)\,.
\end{equation}

From \eqref{eq:orthogonal decomposition} we have the completeness property for $V$: 
$\overline{\bigcup_{j\in\mathbb{Z}}\mathcal{D}_{A}^{j}V}=H^2_G$. 
By Theorem \ref{thm:characterization four} 
the origin is a point of $(G,A^*)$-approximate continuity of $\sigma_V$ if we put  $\sigma_V(0)=1$, 
so by \eqref{eq:spectral function W} the origin is a point of $(G,A^*)$-approximate continuity of $\sigma_W$ if we put $\sigma_W(0)=0$. 
As $W\subseteq H^2_G$ we have by Proposition \ref{prop:properties spectral function} $\sigma_W\leq{\text{\Large $\chi$}}_G$. 
Thus for all $\varepsilon>0$
\begin{displaymath}
\frac{|\{\xi\in A^{*-j}B_1\::\:|\sigma_W(\xi)|\geq\varepsilon\}|}{|A^{*-j}B_1|}=
\frac{|\{\xi\in G\cap A^{*-j}B_1\::\:|\sigma_W(\xi)|\geq\varepsilon\}|}{|A^{*-j}B_1|}\leq
\end{displaymath}
\begin{displaymath}
\leq\frac{|\{\xi\in G\cap A^{*-j}B_1\::\:|\sigma_W(\xi)|\geq\varepsilon\}|}{|G\cap A^{*-j}B_1|}
\xrightarrow[j\to\infty]{}0
\end{displaymath}
By Lemma \ref{lemma:conditions for approximate continuity} 
the origin is a point of $A^*$-approximate continuity of the function $\sigma_W$ if we put $\sigma_W(0)=0$.}
\end{proof}

Given a finite system $\Psi=\{\psi^{\alpha}\}_{\alpha=1}^{N}$ in  $L^2(\mathbb{R}^{d})$, 
we define its associate affine system as
\begin{displaymath}
X(\Psi)=
\{\psi_{(j,k)}^{\alpha}=d_A^{\frac{j}{2}}\,\psi^{\alpha}(A^j\,\cdot-k)\::\:j\in\mathbb{Z},k\in\mathbb{Z}^d,1\leq\alpha\leq N\}\,,
\end{displaymath}
and $W(\Psi)=\overline{span}\{\psi_{(0,k)}^{\alpha}\::\:k\in\mathbb{Z}^d,1\leq\alpha\leq N\}$. 
The symbol $\overline{span}(\mathcal{F})$ denotes the closed linear span in $L^2(\mathbb{R}^{d})$ of the set $\mathcal{F}$. 
Clearly $W(\Psi)$ is shift-invariant and for any $j\in\mathbb{Z}$
\begin{displaymath}
\mathcal{D}_{A}^{j}W(\Psi)=
\overline{span}\{\psi_{(j,k)}^{\alpha}\::\:k\in\mathbb{Z}^d,1\leq\alpha\leq N\}\,.
\end{displaymath}
Fix an $A^*$-invariant set $G$. 
A system $\Psi$ is called a tight frame multiwavelet of $H^2_G$ 
if $X(\Psi)$ is a tight frame of $H^2_G$, that is
\begin{displaymath}
\sum_{\alpha=1}^{N}\sum_{j\in\mathbb{Z}}\sum_{k\in\mathbb{Z}^d}|<f,\psi_{(j,k)}^{\alpha}>|^2
=\left\|\,f\,\right\|^{2}
\quad\forall\:f\in H^2_G\,,
\end{displaymath}
and it is called semiorthogonal if additionally the subspaces $\mathcal{D}_{A}^{j}W(\Psi)$ are mutually orthogonal. 
Applying Proposition \ref{prop:main consequence} to $W(\Psi)$ we get the following:
\begin{theorem} \label{thm:wavelets origin}
If $\Psi$ is a semiorthogonal tight frame multiwavelet of $H^2_G$ 
then for any $1\leq\alpha\leq N$ 
the origin is a point of $A^*$-approximate continuity of $\widehat{\psi^{\alpha}}$ if we put $\widehat{\psi^{\alpha}}(0)=0$.
\end{theorem}
\begin{proof}
{If $\Psi$ is a semiorthogonal tight frame multiwavelet of $H^2_G$ 
then $W(\Psi)$ satisfy the hypothesis of Proposition \ref{prop:main consequence}, 
so the origin is a point of $A^*$-approximate continuity of $\sigma_W$ if we put $\sigma_W(0)=0$. 
Since $\Psi$ is semiorthogonal it is clear that $\Psi$ is a tight frame generator of $W(\Psi)$, so
\begin{displaymath}
\sigma_{W(\Psi)}(\xi)=\sum_{\alpha=1}^N|\widehat{\psi^{\alpha}}(\xi)|^2\,.
\end{displaymath}
Then for any $1\leq\alpha\leq N$ and $\varepsilon>0$ we have
\begin{displaymath}
\frac{|\{\xi\in A^{*-j}B_1\::\:|\widehat{\psi^{\alpha}}(\xi)|\geq\varepsilon\}|}{|A^{*-j}B_1|}\leq
\frac{|\{\xi\in A^{*-j}B_1\::\:|\sigma_W(\xi)|\geq\varepsilon^2\}|}{|A^{*-j}B_1|}
\xrightarrow[j\to\infty]{}0\,,
\end{displaymath}
so by Lemma \ref{lemma:conditions for approximate continuity} 
the origin is a point of $A^*$-approximate continuity of $\widehat{\psi^{\alpha}}$ if we put $\widehat{\psi^{\alpha}}(0)=0$.}
\end{proof}

The system $\Psi$ is called an orthogonal multiwavelet of $H^2_G$ 
if $X(\Psi)$ is an orthonormal basis of $H^2_G$. 
Note that in this case actually $\Psi$ is a semiorthogonal tight frame multiwavelet, 
so the following result holds:
\begin{theorem}
If $\Psi$ is an orthogonal multiwavelet of $H^2_G$ 
then for any $1\leq\alpha\leq N$ 
the origin is a point of $A^*$-approximate continuity of $\widehat{\psi^{\alpha}}$ if we put $\widehat{\psi^{\alpha}}(0)=0$.
\end{theorem}

To the best of our knowledge 
this results are new even for orthonormal wavelets in the one dimensional case with dyadic dilations. 
Note that in \cite{BRS} was shown an example (example 5.6) 
of a dyadic wavelet with essentially unbounded wavelet dimension function whose associated core space $V$ can not be finitely generated. 
Preceding results on the so called cancellation or oscillation property of wavelets, 
namely $\int_{\mathbb{R}}\psi(x)\,dx=0$ (when $\psi\in L^1(\mathbb{R})$), 
can be found in \cite{Ba:89}, \cite{Me:90}, \cite{D:92}, \cite{HW:96}, \cite{W:97}. 
In \cite{BGRW} can be found also some results on the behaviour of the Fourier transform of wavelets at the origin in some cases.

Note that the last two theorems still hold for the case of wavelets of infinite order. 
That is, suppose that $\Psi=\{\psi^{\alpha}\}_{\alpha=1}^{\infty}$ is a sequence in $L^2(\mathbb{R}^{d})$ such that the system
\begin{displaymath}
X(\Psi)=
\{\psi_{(j,k)}^{\alpha}\::\:j\in\mathbb{Z},k\in\mathbb{Z}^d,\alpha=1,2,\ldots\}
\end{displaymath}
is a semiorthogonal tight frame of $H^2_G$. 
This means that we have
\begin{displaymath}
\sum_{\alpha=1}^{\infty}\sum_{j\in\mathbb{Z}}\sum_{k\in\mathbb{Z}^d}|<f,\psi_{(j,k)}^{\alpha}>|^2
=\left\|\,f\,\right\|^{2}
\quad\forall\:f\in H^2_G
\end{displaymath}
and $<\psi_{(j,k)}^{\alpha},\psi_{(l,m)}^{\beta}>=0$ whenever $j\neq l$. 
Then for any $\alpha=1,2,\ldots$ 
the origin is a point of $A^*$-approximate continuity of $\widehat{\psi^{\alpha}}$ if we put $\widehat{\psi^{\alpha}}(0)=0$.

\section{Conclusions}

Consider an expansive dilation $A$ and an $A$-reducing space $H^2_G$ induced by an $A^*$-invariant set of positive measure. 
In the following Theorem we summarize the results in section \ref{sec:completeness}. 
It provides several conditions on the local behaviour at the origin of the spectral function of an $A$-refinable shift-invariant subspace $V$ 
in order to satisfy the so called completeness property $\overline{\bigcup_{j\in\mathbb{Z}}\mathcal{D}_{A}^{j}V}=H^2_G$.

\begin{theorem} \label{thm:characterizations for Areducing}
Let $V\neq\{0\}$ be an $A$-refinable shift-invariant subspace and 
$G\subseteq\mathbb{R}^d$ an $A^*$-set of positive measure such that  $Supp(\sigma_V)\subseteq G$. 
Then the following conditions are equivalent:
\begin{description}
	\item[-] $\overline{\bigcup_{j\in\mathbb{Z}}\mathcal{D}_A^{j}V}=H^2_G\,$;
	\item[-] if $\mathcal{G}$ generates $V$ in the sense that the linear span of the shifts of $\mathcal{G}$ is dense in $V$, 
	then $\bigcup_{\phi\in\mathcal{G}}\bigcup_{j\in\mathbb{Z}}A^{*j}(Supp(\widehat{\phi}))=G$;
	\item[-] for any bounded $E\subseteq G$ of positive measure
	\begin{displaymath}
	\lim_{j\to\infty}\frac{1}{|A^{*-j}E|}\,\int_{A^{*-j}E}\sigma_V(\xi)\,d\xi=1\,;
	\end{displaymath}
	\item[-] the set $Supp(\sigma_V)$ is $A^{*-1}$-absorbing in $G$;
	\item[-] $\lim_{j\to\infty}\sigma_V(A^{*-j}\xi)>0$ for almost every $\xi\in G$;
	\item[-] $\lim_{j\to\infty}\sigma_V(A^{*-j}\xi)=1$ for almost every $\xi\in G$;
	\item[-] $\sigma_V$ is $(G,A^*)$-locally nonzero at the origin;
	\item[-] the origin is a point of $(G,A^*)$-approximate continuity of $\sigma_V$ if we put $\sigma_V(0)=1$.
\end{description}
\end{theorem}

We have also shown that these characterizations have consequences 
for the local behaviour at the origin of the Fourier transforms of wavelet functions. 
Concretely, in Theorem \ref{thm:wavelets origin} it has been proved that 
if $\Psi=\{\psi^{\alpha}\}_{\alpha=1}^{N}$ is a semiorthogonal tight frame multiwavelet of $H^2_G$, 
then the origin is a point of $A^*$-approximate continuity of  $\widehat{\psi^{\alpha}}$ for any $1\leq\alpha\leq N$, if we put  $\widehat{\psi^{\alpha}}(0)=0$. 
This generalizes previous results on the so called cancellation or oscillation property of wavelets (see section \ref{sec:wavelets}).

\section{Acknowledgments}

I would like to thank Professor Kazaros Kazarian for useful discussions and his patience, 
and the anonymous referee for their helpful comments, suggestions and references. 
Author's research was supported by Grants MTM2010-15790 and FPU of Universidad Aut\'onoma de Madrid (Spain).


\end{document}